\documentclass[a4paper,11pt,reqno]{amsart}

\usepackage[T2A]{fontenc}
\usepackage[cp1251]{inputenc}
\usepackage[english]{babel}
\usepackage{amsmath,amsthm,amsfonts,amssymb}
\usepackage{enumerate}
\usepackage{array}

\theoremstyle{plain}
\newtheorem{theorem}{Theorem}

\newtheorem{corollary}{Corollary}
\newtheorem*{corollary*}{Corollary}

\theoremstyle{definition}

\theoremstyle{remark}
\newtheorem{remark}{Remark}
\newtheorem*{remark*}{Remark}
\begin{document}
\title[On fractal phenomena connected with infinite linear IFS]
{On new fractal phenomena\\ connected with infinite linear IFS}
\author[S. Albeverio, Yu. Kondratiev, R. Nikiforov, G. Torbin  ]{Sergio Albeverio$^{1,2,3,4}$, Yuri Kondratiev $^{6,7,9}$ \\ Roman Nikiforov$^{8}$, Grygoriy Torbin$^{9,10}$}

%\date{\currenttime\, \today}

\begin{abstract}
We establish several new fractal and number theoretical phenomena  connected with  expansions which
are generated by infinite linear iterated function systems.
First of all we show that the systems $\Phi$ of cylinders of  generalized L\"uroth expansions are, generally speaking, not faithful for the Hausdorff dimension calculation. Using Yuval Peres' approach, we prove sufficient conditions for the non-faithfulness of such families of cylinders.
 On the other hand, rather general sufficient conditions for the faithfulness of such covering systems are also found. As a corollary of  our main results, we obtain the non-faithfullness of the family of cylinders generated by the classical  L\"uroth expansion.

 Possible infinite entropy of the stochastic vector $Q_\infty$ which determines the metric relations for partitions of  the generalized L\"uroth expansions,  possible non-faithfulness of the family $\Phi(Q_\infty)$ and the absence of general formulae for the calculation of the Hausdorff dimension even for probability measures with independent identically distributed  symbols of generalized L\"uroth expansions are all facts that do not allow  to apply usual probabilistic methods, as well as  methods of dynamical systems to study fractal properties of corresponding subsets of non-normal numbers. Since the `divergent points technique' is not developed for measures generated by infinite IFS, the corresponding methods are also not applicable to solve the problem of the above fractal properties.
  %Methods from \cite{Ols2, Ols1, Ols3}  are also not applicable to solve the problem because of the absence of `divergent points techniques' for the measures %generated by infinite IFS.
  Despite of the above mentioned difficulties, we develop new approach to the study of subsets of $Q_\infty$-essentially non-normal numbers and prove (without any additional restrictions on the stochastic vector $Q_{\infty}$)   that this set is superfractal. This result answers the open problem mentioned in \cite{AKNT2}  and completes the metric, dimensional and topological classification of real numbers via the asymptotic behaviour of frequencies their digits in the generalized L\"uroth expansion.

%
%Based on results related to fine fractal properties of probability measures with independent $Q_\infty$-digits, we study the set of $Q_\infty$-essentially non-normal numbers, i.\,e., real numbers $x$  such that the asymptotic frequency $\nu_i(x)$ of the digit $i$ in the $Q_\infty$-expansion of $x$ does not exist for all digit $i \in \mathbb{N}_0$. We have shown in particular that this set is of the second Baire category and it is of full Hausdorff dimension.

\end{abstract}

\maketitle

$^1$~Institut f\"ur Angewandte Mathematik, Universit\"at Bonn,
Endenicher Allee 60, D-53115 Bonn (Germany); $^2$HCM and ~IZKS, Bonn; $^3$ BiBoS,
Bielefeld--Bonn; $^4$~CERFIM, Locarno; E-mail:
albeverio@iam.uni-bonn.de

$^6$ Fakult\"{a}t  f\"{u}r Mathematik, Universit\"{a}t Bielefeld, Postfach 10 01 31, D-33501,  Bielefeld (Germany) $^7$ BiBoS,
Bielefeld--Bonn; E-mail: kondrat@mathematik.uni-bielefeld.de

$^{8}$~National Pedagogical University, Pyrogova str. 9, 01030 Kyiv
(Ukraine); E-mail: rnikiforov@gmail.com

$^9$~National Pedagogical University, Pyrogova str. 9, 01030 Kyiv
(Ukraine) $^{10}$~Institute for Mathematics of NASU,
Tereshchenkivs'ka str. 3, 01601 Kyiv (Ukraine); E-mail:
torbin@iam.uni-bonn.de (corresponding author)
\medskip

\medskip


\section{Introduction}

There is a rich diversity of systems of numerations for real numbers (see, e.\,g., \cite{DajaniKraaikamp, Schweiger} and references therein).Any such system generates its own metric, dimensional and probabilistic theories. A very important family of such systems consists of expansions generated by iterated function systems (IFS) and admit a dynamical approach to their treatment (\cite{DajaniKraaikamp, Pesin97, Schweiger}). For the case of finite IFS, fractal properties of corresponding attractors and invariant measures were intensively studied during the last 50 years. In particular, it has been shown that methods of the ergodic theory of dynamical systems and statistical physics are of great importance to study fine fractal properties of related sets and measures (see, e.\,g., Bowen \cite{bowen}, Ruelle \cite{ruelle}, Falconer \cite{Fal2} and references therein). An important part of the above mentioned studies is related to the distribution of  frequencies digits in expansions for points in the invariant sets of IFS (see, e.\,g., \cite{APT3, BOS07, BSS02, Besicovitch, Eggleston,  Ols2, Ols1, PT, volkmann59}).

The corresponding theory for infinite IFS is essentially less developed and a lot of new phenomena appear in such a case. It is known, for instance, that even for the infinite \textit{linear} IFS the Hausdorff dimension of the corresponding attractor fails to be the root of the equation $\sum\limits_{i=1}^{\infty} k_i^x = 1 $, where $k_i$ is the similarity ratio of $F_i$ (see, e.\,g., \cite{NT_NZ2008} and references therein).

The most famous expansions with infinite alphabets are the classical continued fractions expansion and the L\"uroth expansion (\cite{DajaniKraaikamp, Schweiger}). Problems  of the dimensional theory of continued fraction expansion are well known (see, e.\,g., \cite{FLM10,Good, Hensley89, Hensley96, Jenkinson, Khi63, luczak}) and can partially be explained by the new phenomenon of fractal non-faithfulness in the dimensional theory of continued fraction expansions which has been discovered recently in \cite{PerTor}. In the present paper we deal with expansions generated by infinite linear IFS $\{F_0, F_1, \dots, F_n, \dots\}$ such that $F_n$ is a similarity transformation with ratio $q_n$, $\sum\limits_{n=0}^{\infty}q_n =1$ and the sequence $\{\sup\limits F_i([0,1])\}$ is strictly monotone. Let
\begin{equation*}\label{delta expansion}
  \Delta_{\alpha_1(x) \alpha_2(x) \dots \alpha_n(x)\dots}
\end{equation*}
be the corresponding expansion of $x\in[0, 1].$  Let us mention that if the above sequence $\{\sup\limits F_i([0,1])\}$ is strictly decreasing and $q_i= \frac{1}{(i+1)(i+2)}$, then we get the classical L\"uroth expansion. For the case of a increasing sequence  we get the  $Q_{\infty}$-expansion. Since metric, dimensional and probabilistic theories are entirely the same for both cases, we will consider only the case where the above sequence increases.
Such an expansion is actually the $f$-expansion  (see, e.\,g., \cite{Everett,Renyi} for details), which is generated by the following strictly increasing continuous  function $f$ defined on $[0, +\infty)$ such that $f(0)=0$ and $f$ increases linearly on each interval $[i, i+1]$ with $f(i+1)-f(i)=q_i, \forall i \in \mathbb{N}_0$.

 One approach to the simplification of the calculation of the Hausdorff dimension consists in some restrictions of admissible coverings. This idea came from Besicovitch's works and has been used by Rogers and Taylor to construct comparable net measures (\cite{R}) as approximations of the Hausdorff measures. In this paper we develop Besicovitch's approach via construction of net coverings which lead to a special  family of net measures which are more general that comparable ones. The first phenomenon we will talk about is connected with the problem of faithfulness and  non-faithfulness of the family of cylinders from the above expansions for the Hausdorff dimension calculation.
To be precise, let us shortly recall that \textit{the $\alpha $-dimensional Hausdorff measure}
of a set $E\subset [0,1]$ with respect to  a given  family of coverings $ \Phi$  is defined by
 $$ H^{\alpha } (E,  \Phi)=  \lim\limits_{\epsilon \to 0 } ~~ \inf\limits_{|E_{j} |\le \epsilon }~ \sum _{j} |E_{j}|^{\alpha }
  = \lim\limits_{\epsilon \to 0 } H_{\epsilon }^{\alpha } (E,\, \Phi ),$$
where the infimum is taken over all at most countable $\epsilon $-coverings $\{ E_{j} \} $ of $E$, $E_{j} \in \Phi$.
  The nonnegative number
\[\dim _{H} (E,\, \Phi)=\inf \{ \alpha :\, \, H^{\alpha } (E,\, \Phi )=0\}\]
\noindent is called the Hausdorff dimension of the set $E\subset [0,1]$ w.r.t. the family $\Phi$.
If $\Phi$ is the family of all subsets of $[0, 1]$, or $\Phi$ coincides with the family of
all closed (open) subintervals of [0,1], then $\dim _{H} (E,\, \Phi)$ is equal to the classical Hausdorff dimension $\dim _{H} (E)$ of the subset $E \subset [0,1]$.

 A fine covering family $\Phi$ is said to be a \textit{faithful family of coverings}~\textit{(non-faithful family of coverings)} for the Hausdorff dimension calculation on $[0,1]$ if $$\dim _{H} (E,\Phi)=\dim _{H} (E), ~~~\forall E\subseteq [0,1]$$ $$(\mbox{resp.} ~~\exists E\subseteq [0,1]: \dim _{H} (E,\Phi)\neq\dim _{H} (E)).$$

It is clear that any family $\Phi$ of  comparable net-coverings (i.\,e., net-coverings which generate comparable net-measures)  is faithful. Conditions for Vitali coverings to be faithful were studied by many authors (see, e.g., \cite{AT2, Bil, Cutler, PT_NZ2003} and references therein). First steps in this direction have been done by A.~Besicovitch (\cite{Besicovitch}), who proved the faithfulness for the family of cylinders of a binary expansion. His result was extended by P.~Billingsley (\cite{Bil}) to the family of $s$-adic cylinders, by M.~Pratsiovytyi (\cite{TuP}) to the family of $Q$-$S$-cylinders, and by S.~Albeverio and G.~Torbin (\cite{AT2}) to  the family of $Q^*$-cylinders for those matrices $Q^*$ whose elements $p_{0k}, p_{(s-1)k}$ are bounded away from zero.

  It is rather paradoxical that  initial  examples of non-faithful families of coverings appeared firstly  in the two-dimensional case (as a result of active studies of self-affine sets during the last decade of XX century (see, e.\,g., \cite{BB})). The family of cylinders of the classical continued fraction expansion can probably be considered as the first (and rather unexpected) example of non-faithful one-dimensional net-family of coverings (\cite{PerTor}).
It is clear that for the family of all cylinders of the $Q_{\infty}$-expansion neither assumptions for a Vitali covering (\cite{Cutler}) nor any other known conditions for the faithfulness do hold.  By using the approach invented by Yuval Peres to prove the non-faithfulness of the family of continued fraction cylinders (\cite{PerTor}), in Section \ref{section:nonfaith} of the present paper we prove the non-faithfulness for the family $\Phi(Q_\infty)$ of cylinders of the $Q_\infty$-expansion with polynomially decreasing elements~$\{q_i\}$ (Theorem 2), which shows that the family of cylinders of the L\"uroth expansion is non-faithful!    On the other hand we give rather general sufficient conditions for $\Phi(Q_\infty)$ to be faithful (Theorem 1). Let us stress that this family of expansions  is the first known one generating faithful as well as non-faithful nets.

\medskip

The second aim of the paper is to study properties of the set of $Q_\infty$-essentially non-normal numbers and obtain a complete metric, topological
 and fractal classification of real numbers via the asymptotic behaviour of frequencies of their $Q_\infty$-digits. To be more precise, let $N_i(x,n)$ be the number of digits  $i$ among the  first $n$  digits of the $Q_\infty$-expansion of $x$. If the limit $\lim\limits_{n\to\infty} \frac{N_i(x,n)}{n}$ exists, then its value $\nu_i(x)$ is said to be \emph{the asymptotic frequency of the digit  `$i$' in the $Q_\infty$-expansion of  $x$}. By the law of large numbers, for Lebesgue almost all real numbers from the unit interval we have $\nu_i (x) = q_i,\: \forall i \in\mathbb{N}_0.$

The set
$
N(Q_{\infty}) = \left\{x : \; \forall i \in \mathbb{N}_0, \,\lim\limits_{n\to\infty}
\frac{N_i(x,n)}{n} = q_i \right\}
$
is said to be the set of  $Q_\infty$-normal numbers.

The set
$
W(Q_{\infty}) = \left\{x : \; \forall i \in \mathbb{N}_0, \, \lim\limits_{n\to\infty}
\frac{N_i(x,n)}{n} \;\text{exists}\right\} \bigcap \overline{N(Q_\infty)}
$
is said to be the set of  $Q_\infty$-quasi-normal numbers, where $\overline{N(Q_\infty)}=[0, 1)\setminus N(Q_\infty).$

The set
$
D(Q_\infty) = \Big\{ x:  \; \exists i_0 \in N_0,  \, \liminf\limits_{n\to\infty} \frac{N_{i_0}(x,n)}{n} < \limsup\limits_{n\to\infty} \frac{N_{i_0}(x,n)}{n}    \Big\}$ is said to be the set of  $Q_\infty$-non-normal numbers.

    The set
$$ \begin{aligned} P(Q_\infty)=\Big\{x:\quad & \exists i_0 \in \mathbb{N}_0,  \, \liminf\limits_{n\to\infty} \frac{N_{i_0}(x,n)}{n} < \limsup\limits_{n\to\infty} \frac{N_{i_0}(x,n)}{n},\\
& \exists i_1 \in \mathbb{N}_0, \,  \liminf\limits_{n\to\infty} \frac{N_{i_1}(x,n)}{n} = \limsup\limits_{n\to\infty} \frac{N_{i_1}(x,n)}{n} \Big\}
\end{aligned}
$$
is said to be the set of  $Q_\infty$-partially non-normal numbers.

    The set
$  L(Q_\infty)=\Big\{x: \; \forall i\in \mathbb{N}_0, \, \liminf\limits_{n\to\infty} \frac{N_{i}(x,n)}{n} < \limsup\limits_{n\to\infty} \frac{N_{i}(x,n)}{n}  \Big\}
$
is said to be the set of  $Q_\infty$-essentially non-normal numbers.

For the case of the $s$-adic expansion ($s>1,~ s \in \mathbb{N}$) and some generalizations for expansions with finite alphabets several approaches have been developed to study properties of the corresponding sets.  In 1995 it has been shown (\cite{PT}) that the set of non-normal numbers w.r.t. base $s$ is superfractal (i.\,e., a set of zero Lebesgue measure and of full Hausdorff dimension). By using different approaches, in \cite{APT3} and in \cite{BSS02} it has been proven that the set of real numbers which are essentially non-normal w.r.t. base $s$ is also superfractal.  Another powerful approach (the so-called `divergence points' approach) was developed in a series of papers by L.~Olsen, S.~Winter, N.~Snigereva, I.~S.~Baek, A.~Bisbas (see, e.\,g., \cite{BOS07, Bisbas12, Ols2, Ols1,  Ols4, Ols3} and references therein). From Volkmann results \cite{volkmann59} it follows that this set is residual, and, therefore, other three subsets $N(Q_\infty), W(Q_\infty), P(Q_\infty)$ are of the first Baire category.

For the case of infinite IFS the situation is essentially more complicated. In \cite{FLMW10} the authors pointed out that ``the difference from earlier works is the (countable) infinity of the alphabet, from which comes a particular phenomenon that the formal variational principle does not hold as in the case of compact dynamics and \dots the thermodynamic formalism does not work as one wishes. In fact, we do not know whether there always exist a Gibbs measure, to be suitable defined, on the set in question, which has the dimension of the set.''  In the same paper \cite{FLMW10} the authors  also generalized results from \cite{BI09} and studied fractal properties of Besicovitch--Eggleston sets, which are subsets of  $W(Q_{\infty})$,  found   explicit formulae for their Hausdorff dimension  and stressed significant differences from those in the $s$-adic expansion.

 Possible infinite entropy of the stochastic vector $Q_\infty$, possible non-faithfulness of the family $\Phi(Q_\infty)$ and the absence of general formulae for the calculation of the Hausdorff dimension even for probability measures with independent \textit{identically distributed}  $Q_\infty$-symbols (this is still an open problem (see, e.\,g., \cite{BI09, NT_NZ2008, NT_TVIMS12})) do not allow us to apply methods from \cite{APT3} to study fractal properties of the above defined sets. Let us also mention that methods from \cite{Ols2, Ols1, Ols3}  are also not applicable to solve the problem because of the absence of `divergent points techniques' for the measures generated by infinite IFS.

Despite the above mentioned difficulties, in the last section of the present work we develop a new approach to the study of subsets of $Q_\infty$-essentially non-normal numbers and prove (without any additional restrictions on the stochastic vector $Q_{\infty}$ like convergence of the series $\sum\limits_{j=0}^\infty\frac{\ln^2 q_j}{2^j}$ (see, e.g., \cite{AKNT2}), finiteness of the entropy or faithfulness of the corresponding family $\Phi(Q_\infty)$)  that this set is superfractal. This result answers the open problem mentioned in \cite{AKNT2}  and completes the metric, dimensional and topological classification of real numbers via the asymptotic behaviour of frequencies of their digits in the generalized L\"uroth expansion:

\vskip 0.3 cm
\begin{tabular}{|c|c|c|c|c|} \hline
& Lebesgue measure& Hausdorff dimension & Baire category\\
\hline
$N(Q_\infty)$&1&1& first \\
\hline
$W(Q_\infty)$&0&1& first \\
\hline
$P(Q_\infty)$&0&1& first \\
\hline
$L(Q_\infty)$&0&1& second \\
\hline
\end{tabular}

\section{On faithful and  non-faithful covering systems generated by $Q_{\infty}$-expansion}\label{section:nonfaith}

In this section we shall demonstrate that fractal non-faithful phenomena for one-dimen\-sional nets can appear even for those nets which  are generated by \textit{linear} infinite IFS. Non-faithfulness for the Hausdorff dimension calculation of the family of cylinders for the classical L\"uroth expansion is a  direct consequence of our results.

For the convenience of the reader we give a geometrical explanation of the above mentioned  $Q_\infty$-expansion.
 Given a $Q_\infty$-vector we consecutively perform decompositions of the semi-interval $[0,1)$ in the following way.

Step 1. We decompose $[0,1)$ (from the left to the right) into the union of semi-intervals $\Delta_{\alpha_1}$, $\alpha_1\in \mathbb{N}_0$ (without common points) of the length $\left|\Delta_{\alpha_1}\right| =q_{\alpha_1}$,
\begin{equation*}
\left[ 0,1 \right) =\bigcup_{\alpha_1=0}^\infty
\Delta_{\alpha_1}.
\end{equation*}
Each interval $\Delta_{\alpha_1}$ is called a 1-rank cylinder (basic interval).

 Step $n \geq 2$. We decompose (from the left to the right) each $(n-1)$-rank cylinder $\Delta_{\alpha_1\dots \alpha_{n-1}}$ into the union of $n$-rank cylinders $\Delta_{\alpha_1\dots \alpha_n}$ (without common points)
\begin{equation*}
\Delta_{\alpha_1\dots \alpha_{n-1}}=\bigcup_{\alpha_n =0}^\infty
\Delta_{\alpha_1\dots \alpha_n},
\end{equation*}
whose lengths
\begin{equation}\label{eq:diameterofQinftycylinder}
\left| \Delta_{\alpha_1\alpha_2\dots \alpha_n}\right|= q_{\alpha_1}\cdot
q_{\alpha_2}\cdots q_{\alpha_n}=\prod_{i=1}^nq_{\alpha_i}
\end{equation}
are related as follows
\begin{equation*}
\left| \Delta_{\alpha_1\dots \alpha_{n-1}0}\right| :\left|
\Delta_{\alpha_1\dots \alpha_{n-1}1}\right| :\cdots:\left| \Delta_{\alpha_1\dots \alpha_{n-1} \alpha_n}\right| :\cdots
=q_{0}:q_{1}:\cdots :q_{\alpha_n }:\cdots.
\end{equation*}

It is clear that  any sequence of indices $\{ \alpha_n \}$
generates the corresponding sequence of  embedded cylinders
\begin{equation*}
\Delta_{\alpha_1}\supset \text{ }\Delta
_{_{\alpha_1\alpha_2}}\supset \cdots  \supset \Delta
_{_{\alpha_1\alpha_2\dots \alpha_n}}\supset \cdots
\end{equation*}
and  there exists a unique point $x\in \left[ 0,1\right)$ belonging to all of them.

Conversely, for any point $x\in \left[ 0,1\right) $ there exists a unique sequence of embedded cylinders $\Delta_{\alpha_1}\supset
\Delta_{\alpha_1\alpha_2}\supset \ldots\supset $
$\Delta_{\alpha_1\alpha_2 \dots \alpha_n}\supset \dots$ containing $x$, i.\,e.,
\begin{equation*}
x=\bigcap_{n=1}^\infty
\Delta_{\alpha_1\dots \alpha_n}=\bigcap_{n=1}^\infty
\Delta_{\alpha_1(x)\dots \alpha_n(x)}=: \Delta_{\alpha_1(x)\dots \alpha_n(x)\ldots};
\end{equation*}
$$x=\Delta_{\alpha_1(x)\dots \alpha_n(x)\ldots}.$$
The  expression is said to be the $Q_{\infty}$-expansion for $x$.  Real numbers which are end-points of n-th rank cylinders are said to be $Q_\infty$-rational, and their $Q_\infty$-expansion contains only finitely many non-zero digits. In the sequel,
$\Phi = \Phi(Q_{\infty})$ will be the family of all possible cylinders of
the $Q_\infty$-partition of the semi-interval $[0,1)$, i.\,e.,
\begin{equation*}
\Phi=
\{ E: E= \Delta_{\alpha_1 \dots \alpha_n}, ~~
 ~ \alpha_i \in\mathbb{N}_0, ~ i= 1,2, \dots,n ; ~~~ ~~~ n \in\mathbb{N}\}.
\end{equation*}

Before we present the new phenomena related to the non-faithfulness of the nets generated by infinite IFS, we give rather general sufficient conditions for nets $\Phi$ generated by the $Q_{\infty}$-expansion to be faithful.
\begin{theorem}
Let $Q_\infty=(q_0, \dots, q_i,\dots)$  be a stochastic vector such that for any $\alpha>0$ there exists a constant $c=c(\alpha)$ with
\begin{equation}\label{eq:condfaithnew}
\sum_{k=i+1}^\infty q_k^\alpha\leq c(\alpha)q_i^\alpha, \;\forall i\in \mathbb{N}_0.
\end{equation}
Then the family $\Phi$ is faithful for the Hausdorff dimension calculation on the unit interval.
\end{theorem}

\begin{proof}
For the calculation of the Hausdorff dimension of a set $E\subset [0, 1)$ we may consider only coverings of intervals $[a_j, b_j)$ such that $a_j\in A,$ $b_j\in A$ and the set $A$ is an everywhere dense set on $[0, 1].$ Let us choose $A$ to be the set of all $Q_\infty$-rational numbers.

For a given set $E,$ $\alpha>0,$ $\epsilon>0$ let us choose $\delta\in (0, \alpha),$ and let $\{E_j\}$ be an $\epsilon$-covering of $E$ by above intervals $E_j=[a_j, b_j),$ $a_j\in A,$ $b_j\in A.$

For any $j\in \mathbb{N}$ there exists the cylinder $\Delta_{\alpha_1\alpha_2\dots\alpha_{n_j}}$ of maximal rank that contains $[a_j, b_j)$ (i.\,e., any cylinder of $(n_j+1)$-th rank does not contain $[a_j, b_j)$). So, the $a_j$
have the following $Q_\infty$-expansion:
\begin{equation*}
a_j=\Delta_{\alpha_1\alpha_2\dots\alpha_{n_j}\alpha_{n_j+1}\dots\alpha_{n_j+l_j}00\dots},
\end{equation*}
where $\alpha_k=\alpha_k(a_j).$
Then, the point
$c_j=\Delta_{\alpha_1\dots\alpha_{n_j}(\alpha_{n_j+1}+1)00\dots}$
belongs to~$[a_j, b_j).$

Let $\beta_k\!:=\alpha_{n_j+k}(a_j)$ be the $(n_j+k)$-th digit of the $Q_\infty$-expansion of $a_j,$ $k\in\{1, 2, \dots, l_j\}.$
To cover $E_j$ by cylinders we consider the coverings of $[a_j, c_j)$ and $[c_j, b_j)$ separately.

If the cylinder $\Delta_{\alpha_1\dots\alpha_{n_j}(\beta_1+1)}$ belongs to $[c_j, b_j),$ then $[c_j, b_j)$ can be covered by cylinders
\begin{equation*}
\Delta_{\alpha_1\dots\alpha_{n_j}(\beta_1+i)}, \; i\in\mathbb{N}.
\end{equation*}

The $\alpha$-volume of this covering is equal to

\begin{eqnarray}
\sum_{i=1}^\infty|\Delta_{\alpha_1\dots\alpha_{n_j}(\beta_1+i)}|^\alpha\leq|\Delta_{\alpha_1\dots\alpha_{n_j}(\beta_1+1)}|^\alpha+\sum_{i=2}^\infty|\Delta_{\alpha_1\dots\alpha_{n_j}(\beta_1+i)}|^\alpha\leq\nonumber\\
\leq(1+c)|\Delta_{\alpha_1\dots\alpha_{n_j}(\beta_1+1)}|^\alpha\leq(1+c)|E_j|^\alpha.\nonumber
\end{eqnarray}
If the cylinder $\Delta_{\alpha_1\dots\alpha_{n_j}(\beta_1+1)}$ does not belong to $[c_j, b_j),$
then there exists a positive integer $t_j$ such that the cylinder
\begin{equation*}\label{eq:cylindercontainer}
\Delta_{\alpha_1\dots\alpha_{n_j}(\beta_1+1)\underbrace{0\dots0}_{t_j}}
\end{equation*}
contains $[c_j, b_j)$ and the cylinder $\Delta_{\alpha_1\dots\alpha_{n_j}(\beta_1+1)\underbrace{0\dots0}_{t_j+1}}$ is contained in $[c_j, b_j).$

In such a case the cylinder \eqref{eq:cylindercontainer} covers $[c_j, b_j)$ and its diameter does not exceed $\frac{1}{q_0}|E_j|.$
So, $[c_j, b_j)$ can be covered by cylinders from $\Phi$ such that the corresponding $\alpha$-volume does not exceed $B_0^\alpha\,|E_j|^\alpha,$ where $B_0=\max\{1+c, \frac{1}{q_0}\}.$

The most difficult point in the proof is to find small enough (in the sense of $\alpha$-volume) covering of $[a_j, c_j)$ by cylinders from~$\Phi.$

The set $[a_j, c_j)$ contains non-intersecting cylinders of the following ranks.

\begin{eqnarray*}
&\text{Rank}\, (n_j+2): \quad &\Delta_{\alpha_1\dots\alpha_{n_j}\beta_1 i}, \quad i>\beta_2.\\
&&\vdots\\
&\text{Rank}\, (n_j+l_j-1): \quad &\Delta_{\alpha_1\dots\alpha_{n_j}\beta_1\dots \beta_{l_j-2} i}, \quad i>\beta_{l_j-1}.\\
&\text{Rank}\, (n_j+l_j): \quad &\Delta_{\alpha_1\dots\alpha_{n_j}\beta_1\dots \beta_{l_j-1} i}, \quad i\geq \beta_{l_j}.
\end{eqnarray*}

We recall that $a_j=\inf \Delta_{\alpha_1\dots\alpha_{n_j}\beta_1\beta_2\dots \beta_{l_j-1}\beta_{l_j}}.$

So, $[a_j, c_j)$ is split into the union of the above cylinders of ranks $(n_j+2), \dots, (n_j+l_j).$
The $\alpha$-volume of cylinders of $(n_j+k)$-th rank ($k<l_j$) is equal to
$$\sum_{i=\beta_k+1}^\infty|\Delta_{\alpha_1\dots\alpha_{n_j}\beta_1\dots \beta_{k-1}i}|^\alpha\leq(1+c)|\Delta_{\alpha_1\dots\alpha_{n_j}\beta_1\dots \beta_{k-1}(\beta_k+1)}|^\alpha$$
and
$$\sum_{i=\beta_{l_j}}^\infty|\Delta_{\alpha_1\dots \alpha_{n_j}\beta_1\dots \beta_{l_j-1} i}|^\alpha\leq (1+c)|\Delta_{\alpha_1\dots\alpha_{n_j}\beta_1\dots \beta_{l_j-1} \beta_{l_j}}|^\alpha.$$

Thus $[a_j, c_j)$ can be covered by a countable number of above cylinders, and the corresponding $\alpha$-volume does not exceed the value
\begin{equation}\label{eq:sumalphacovering}
\mathrm{v}_j(\alpha)=(1+c)\big(|\Delta_{\alpha_1\dots\alpha_{n_j}(\beta_1+1)}|^\alpha+\dots+|\Delta_{\alpha_1\dots\alpha_{n_j}\beta_1\dots \beta_{l_j}}|^\alpha\big).
\end{equation}

Unfortunately in the latter sum the distribution of lengths of cylinders can be arbitrary. In particular, even the cylinder of maximal rank can be of maximal length (one can check this by considering the case where $q_i=\frac{a}{2^{2^i}}$), and the sequence $\{l_j\}$ is, generally speaking, unbounded.

For a given $l_j\in\mathbb{N}$ one can choose $d_j\in\mathbb{N}$ such that $2^{d_j-1}<l_j\leq 2^{d_j}.$ Let $\mathbf{q}\!:=\max_i(q_i).$ Then the above sum~\eqref{eq:sumalphacovering} can be rewritten as follows:

\begin{align*}
\mathrm{v}_j(\alpha)&=(1+c)\Big[\sum_{i=1}^{d_j-1}\sum_{2^{i-1}<k\leq 2^i}|\Delta_{\alpha_1\dots\alpha_{n_j}\beta_1\dots \beta_{k-1}(\beta_k+1)}|^\alpha+\\
&+\sum_{2^{d_j-1}<k\leq l_j-1}|\Delta_{\alpha_1\dots\alpha_{n_j}\beta_1\dots \beta_{k-1}(\beta_k+1)}|^\alpha+|\Delta_{\alpha_1\dots\alpha_{n_j}\beta_1\dots \beta_{l_j}}|^\alpha\Big]\leq\\
&\leq (1+c)|E_j|^{\alpha-\delta}\sum\limits_{k=0}^{d_j-1}2^k (\mathbf{q}^\delta)^{2^k}<(1+c)|E_j|^{\alpha-\delta}\sum_{s=1}^\infty s(\mathbf{q}^\delta)^{\frac s2}(\mathbf{q}^\delta)^{\frac s2},
\end{align*}
for any $\delta\in (0, \alpha).$

It is clear that there exists a constant $W(\delta)$ such that $s(\mathbf{q}^\delta)^{\frac s2}\leq W(\delta),$ $\forall s\in\mathbb{N}.$
Therefore,
\begin{align*}
\mathrm{v}_j(\alpha)\leq (1+c)|E_j|^{\alpha-\delta} W(\delta)\sum_{s=1}^\infty(\mathbf{q}^\delta)^{\frac s2}=\frac{(1+c)W(\delta)\mathbf{q}^{\frac \delta 2}}{1-\mathbf{q}^{\frac \delta 2}}|E_j|^{\alpha-\delta}.
\end{align*}

So, for a given $E_j=[a_j, b_j)$ there exists a countable family of cylinders that cover $E_j$ and whose $\alpha$-volume does not exceed
$K(\alpha, \delta)|E_j|^{\alpha-\delta}$ with $$K(\alpha, \delta)\!:=B_0+\frac{(1+c(\alpha))W(\delta)\mathbf{q}^{\frac \delta 2}}{1-\mathbf{q}^{\frac \delta 2}}.$$

Therefore, $\forall\alpha>0,$ $\forall\delta\in(0, \alpha),$ $\forall E\subset[0, 1)$ we have
$$H^\alpha(E)\leq H^\alpha(E, \Phi)\leq K(\alpha, \delta)H^{\alpha-\delta}(E).$$

Hence, $\dim_H(E, \Phi)\leq \dim_H (E)+\delta,$ $\forall \delta\in (0, \alpha),$ which proves the equality $\dim_H(E, \Phi)=\dim_H(E).$
\end{proof}

\begin{corollary}
If $\varlimsup\limits_{n\to\infty}\frac{q_{n+1}}{q_n}<1,$ then $\Phi$ is faithful.
\end{corollary}

\begin{remark}
  From the latter theorem it follows that $\Phi$ is faithful for the case where the series $\sum\limits_{i=1}^\infty q_i$ converges rather quickly in the sense of \eqref{eq:condfaithnew}.
\end{remark}

The following rather unexpected  Theorem shows that the fine covering system generated by the $Q_{\infty}$-expansion is not necessarily faithful, which shows a new phenomenon related to infinite IFS.

\begin{theorem}\label{theorem on simple sufficient cond for non-faithfulness}
 If there exist constants $m_0>1,$ $A>0$ and $B>0$ such that
 \begin{equation}\label{cond for faithfulness}
   \frac{A}{i^{m_0}} \leq q_i \leq \frac{B}{i^{m_0}}, \forall i \in\mathbb{N},
 \end{equation}
 then  the fine covering system generated by the $Q_{\infty}$-expansion is non-faithful.
 \end{theorem}

    \begin{proof}

%     спочатку доведемо допоміжну нерівність. Використовуючи геометричний зміст визначеного інтеграла отримаємо
%          \begin{eqnarray}\label{nerivnist_z_integralom_Q_infty}
%            \int_a^{b+1}\frac{1}{x^n}dx\leq&\sum\limits_{i=a}^{b}\frac{1}{x^n}&\leq\int_{a-1}^b\frac{1}{x^n} dx,\nonumber\\
%            \frac{1}{n-1}\bigg(\frac{1}{a^{n-1}}-\frac{1}{(b+1)^{n-1}}\bigg)\leq&
%            \sum\limits_{i=a}^b\frac{1}{x^n}&\leq\frac{1}{n-1}\bigg(\frac{1}{(a-1)^{n-1}}-
%            \frac{1}{b^{n-1}}\bigg),\nonumber\\
%            \frac{1}{n-1}\bigg(\frac{1}{a^{n-1}}-\frac{1}{(b+1)^{n-1}}\bigg)\leq &\sum\limits_{i=a}^b\frac{1}{x^n}&\leq
%            \frac{2^{n-1}}{a^{n-1}}.
%          \end{eqnarray}

Let us construct a set whose Hausdorff dimension is equal to zero, but the corresponding value of the Hausdorff dimension w.r.t. the fine covering system $\Phi$ is at least~$\frac{1}{m_0}$.

To this end we shall investigate properties of  the Cantor-like set
\begin{equation*}
C[Q_\infty, \{V_k\}]=\{x:\quad x=\Delta_{\alpha_1(x)\ldots\alpha_k(x)\ldots}, \alpha_k(x)\in V_k\},
\end{equation*}
          where $$V_k=\{i:\quad i\in\mathbb{N}, l_{2k}\leq i\leq l_{2k+1}\}$$
          and the sequences $\{l_{2k}\}$, $\{l_{2k+1}\}$, $\{M_{k}\}$ are defined recursively as follows:

$$
M_0\!:=1,  ~~l_{2}=2 M_{0}=2, ~~l_{3}=(l_{2}+1)^2=9;
$$
\begin{equation*}
              ~~l_{2k}=\big(2^k M_{k-1}\big)^k, l_{2k+1}=(l_{2k}+1)^2,  M_{k-1}=l_3\cdot l_5\cdot\ldots\cdot l_{2k-1}, \forall k\in\mathbb{N}.
\end{equation*}
          Firstly, let us prove that $\dim_H (C[Q_\infty,\{V_k\}])=0$. From the construction of the set  $C[Q_\infty,\{V_k\}]$ it follows that for any $k \in\mathbb{N}$ this set can be covered by $M_k$ cylinders of rank~$k.$

          On the other hand, this set can  be covered by $M_{k-1}$ intervals $\nabla_{\alpha_1\ldots\alpha_{k-1}}$, which are unions of cylinders of rank $k$, i.\,e.,
          $$\nabla_{\alpha_1\ldots\alpha_{k-1}}=\bigcup\limits_{i=l_{2k}}^{l_{2k+1}}
          \Delta_{\alpha_1\ldots\alpha_{k-1}i}.$$
                      It is not hard to see that
                     $$  |\nabla_{\alpha_1\ldots\alpha_{k-1}}|
                          \leq \sum\limits_{i=l_{2k}}^{l_{2k+1}}\frac{B}{i^{m_0}} \leq \int\limits_{l_{2k}-1}^{l_{2k+1}}\frac{B}{x^{m_0}}dx  \leq  \frac{2^{m_0-1}B}{l_{2k}^{m_0-1}}.
          $$

          Therefore, for any positive $\alpha$ the $\alpha$-volume of the latter covering of the set $C[Q_\infty,\{V_k\}]$ does not exceed the value $$\frac{M_{k-1} 2^{\alpha (m_0-1) }B^{\alpha}}{(2^k M_{k-1})^{\alpha k (m_0-1)}},$$
          which tends to zero as $k$ tends to infinity.

          So, \begin{equation*}
            \dim_H(C[Q_\infty,\{V_k\}])=0.
          \end{equation*}

            Now we show that $\dim_H(C[Q_\infty,\{V_k\}],\Phi)>0.$

             Let  $\xi$ be the random variable with independent $Q_{\infty}$-digits, which is uniformly distributed on the set  $C[Q_\infty,\{V_k\}]$, i.\,e., $\xi$ is of the form
            $$\xi=\Delta_{\xi_1\ldots\xi_k\ldots},$$
             where $\xi_k$ are independent random variables taking values
             $$l_{2k},~~~  l_{2k}+1,~~ \ldots,   l_{2k+1}$$ with probabilities
                  $$\frac{1}{\gamma_k}\cdot q_{l_{2k}},~~ \frac{1}{\gamma_k}\cdot q_{l_{2k}+1},~  \ldots~,  \frac{1}{\gamma_k}\cdot
                  q_{l_{2k+1}}$$
                   correspondingly. The normalizing constants $\gamma_k$ are determined by $$\sum\limits_{i=l_{2k}}^{l_{2k+1}}\frac{1}{\gamma_k} q_i=1.$$ Then, taking into account assumptions (\ref{cond for faithfulness}), we get
                $$\gamma_k \geq A\cdot\sum\limits_{i=l_{2k}}^{l_{2k+1}}\frac{1}{i^{m_0}}=
                   A\cdot\left(\frac{1}{l_{2k}^{m_0}}+\sum\limits_{i=l_{2k}+1}^{l_{2k+1}}\frac{1}{i^{m_0}}\right)\geq $$
                    \begin{equation*}
\geq A\cdot\left(\frac{1}{l_{2k}^{m_0}}+\int_{l_{2k}+1}^{l_{2k+1}+1}\frac{1}{x^{m_0}}dx\right) \geq\frac{A}{(m_0-1)\cdot(l_{2k}+1)^{m_0-1}}.
                   \end{equation*}

             Let $\mu_\xi$ be the probability measure of $\xi$. Then
               \begin{eqnarray*}
                  \mu_\xi(\Delta_{\alpha_1\ldots
                  \alpha_k})&=&\frac{1}{\gamma_1}q_{\alpha_1}\cdot\ldots
                  \cdot\frac{1}{\gamma_k}q_{\alpha_k}\leq  \frac{1}{\gamma_k}\cdot q_{\alpha_k}
                  \leq\frac{D\cdot l_{2k}^{m_0-1}}{\alpha_k^{m_0}},
               \end{eqnarray*}
             where $D=\frac{B(m_0-1)2^{m_0-1}}{A}$.

             On the other hand we have
               \begin{eqnarray*}
                   |\Delta_{\alpha_1\alpha_2\ldots \alpha_k}|\geq
                   \frac{A^k}{M^{m_0}_{k-1}}\cdot\frac{1}{\alpha_k^{m_0}}.
               \end{eqnarray*}
             So, for any $x\in C[Q_\infty,\{V_k\}]$ and for  any  $\alpha \in (0,  \frac {1}{m_0})$
              \begin{eqnarray*}
                  \frac{\mu_\xi(\Delta_{\alpha_1(x)\alpha_2(x)\ldots \alpha_k(x)})}{|\Delta_{\alpha_1(x)\alpha_2(x)\ldots
                  \alpha_k(x)}|^\alpha}&\leq& \frac{D\cdot
                  l_{2k}^{m_0-1}}{\alpha_k^{m_0}}:\left(\frac{A^k}{M^{m_0}_{k-1}\alpha_k^{m_0}}\right)^\alpha\leq\\
                  &\leq&
                  \frac{D}{M_{k-1}^{k(1-\alpha m_0)-\alpha m_0}\cdot(A^{\alpha }2^{k(1-\alpha m_0)})^k}.
              \end{eqnarray*}
                         So,
              \begin{equation*}
                \lim_{k\to\infty}\frac{\mu_\xi(\Delta_{\alpha_1(x)\alpha_2(x)\ldots \alpha_k(x)})}{|\Delta_{\alpha_1(x)\alpha_2(x)\ldots
                \alpha_k(x)}|^\alpha}=0,
              \end{equation*}
            for all $\alpha \in (0,  \frac{1}{m_0})$ and $x\in C[Q_\infty,\{V_k\}],$
            and, therefore,
              \begin{equation*}
                 \dim_H(C[Q_\infty,\{V_k\}], \Phi) \geq \frac{1}{m_0}\neq 0 =\dim_H(C[Q_\infty,\{V_k\}]),
              \end{equation*}
which proves the theorem.
\end{proof}

\begin{remark}\label{remark on cont fractions}
  The  proof of the latter theorem is  based on the  method which was invented by Yuval Peres (see \cite{PerTor} for details) to prove the non-faithfulness of the family of continued fraction cylinders.
\end{remark}

\begin{remark}
  One can prove that $\dim_H(C[Q_\infty,\{V_k\}], \Phi)=\frac{1}{m_0}.$
\end{remark}

 \begin{corollary}
   If $q_i = \frac{1}{P_{m_0}(i)}$ for all $i$ which are large enough and for a polynomial $P_{m_0}(x)$ of a degree $m_0>1$, then the corresponding  fine covering family generated by this $Q_{\infty}$-expansion is non-faithful.
 \end{corollary}

\begin{corollary}
The family of cylinders generated by the L\"{u}roth expansion \cite{DajaniKraaikamp, Schweiger} is non-faithful.
\end{corollary}

 \textbf{Open problem.}
   The problem of finding necessary and sufficient conditions for the faithfulness resp. non-faithfulness of a family of  cylinders of a given $Q_{\infty}$-expansion is still open.

\section{Superfractality of the set of $Q_{\infty}$-non-normal numbers}\label{Section 5}

The main aim of this section is to prove the superfractality of the set of $Q_\infty$-essentially non-normal numbers without any additional restrictions on the stochastic vector $Q_{\infty}$.

\begin{theorem}
 The set $L(Q_\infty)$ of $Q_\infty$-essentially non-normal numbers is  of full Hausdorff dimension.
\end{theorem}
\begin{proof}
   The main idea of the proof is rather clear:  to construct a countable family of subsets from $L(Q_\infty)$ whose Hausdorff dimension can be arbitrarily close to unity.
  Possible infinite entropy of the stochastic vector $Q_\infty$, possible non-faithfulness of the family $\Phi(Q_\infty)$ and the absence of general formulae for the calculation of the Hausdorff dimension for probability measures with independent $Q_\infty$-symbols (this is still an open problem (see, e.g., \cite{NT_TVIMS12})) do not allow us to apply methods from \cite{AKNT2, APT3} to construct such a family. Methods from \cite{Ols2, Ols1}  are also not applicable to solve the problem because of the absence of `divergent points techniques' for the measures generated by infinite IFS.

To overcome the above problems, for any
 stochastic vector $Q_\infty$ we shall construct a two-parametric family of subsets $T_{s, l}=T_{s,l} (Q_\infty)$ with a desired properties and apply the developed probabilistic and Hausdorff dimension   techniques to estimate~$\dim_H(T_{s, l}).$

   So,  let $s$ and $l>2$ be fixed positive integers. For a given stochastic vector $Q_\infty = (q_0, q_1, ..., q_i, ...)$ let $m_0\!:=1$ and $m_i\!:=[\ln^2 q_i]\cdot 2^i,$ $\forall i\in\mathbb{N},$ and let $T_{s, l}$ be the set of real numbers whose $Q_\infty$-symbols can be separated into groups such that the $k$-th group is of the following structure
$$\alpha_{k, 1}\dots\alpha_{k, sR_k}\underbrace{0\dots 0}_{R_k} \,\underbrace{1\dots 1}_{\frac{R_k}{m_1}}\ldots \underbrace{i\dots i}_{\frac{R_k}{m_i}}\dots\underbrace{(k-2)\dots(k-2)}_{\frac{R_k}{m_{k-2}}}\,\underbrace{(k-1)\dots(k-1)}_{\frac{R_km_k}{m_{k-1}(m_k-1)}},$$
where $R_1\!:=m_1,$ and $R_k\!:=m_1m_2\cdot\ldots\cdot m_{k-1}(m_k-1),$ $\forall k\geq 2$ and symbols $\alpha_{k, 1},$ $\alpha_{k, 2}, \ldots,$ $\alpha_{k, sR_k}$ can be chosen independently from the set $\{0, 1, \dots, l-1\}.$

 Let us denote by $\mathrm{Fix}(j)$ the set of  numbers of  positions of the fixed digit `$j$' in the $Q_\infty$-expansion of~$x\in T_{s, \; l}.$

 Let
 $$\mathrm{Fix}=\bigcup_{j=0}^\infty \mathrm{Fix}(j);\;\; \;\; \mathrm{Flex}=\mathbb{N}\setminus\mathrm{Fix}.$$

  Then the set $T_{s, \; l}$ can be defined by
  \begin{eqnarray*}
    T_{s, \; l}=\{x:\; &&x=\Delta_{\alpha_1(x)\alpha_2(x)\ldots\alpha_n(x)\ldots};\\
    &&\alpha_n(x)=j \;\text{for all} \;n\in\mathrm{Fix}(j), \; j\in\mathbb{N}_0;\\
    &&\alpha_n(x)\in\{0, 1, \ldots, l-1\} \;\text{for all} \;n\in\mathrm{Flex}\}.
  \end{eqnarray*}

  Firstly let us show that $T_{s,\; l}\subset L(Q_\infty).$ To this end we must prove that  for any $x\in T_{s, \; l}$ and for any digit $i \in \mathbb{N}_0$ the limit $\lim\limits_{n\to\infty}\frac{N_i(x, n)}{n}$ does not exist. Since the proof is very similar for all digits we shall explain it only for the digit $0$.
    Let $n^{0}_k$ be the number of the position at which the series of fixed zeros of the $k$-th group is ended. From the construction of the set $T_{s,l}$ it follows that
  $$
      n^{0}_k = (s+1) \prod_{i=1}^k m_i + \sum\limits_{j=1}^{k-2} \Bigg(\frac{\prod\limits_{i=1}^{k-1} m_i}{m_j} \Bigg).
  $$

Let $n^1_k$ be the number of the position at which the series of fixed ones of the $k$-th group is ended, i.\,e., $n^1_k=n^0_k+\frac{R_k}{m_1}.$ Then
    \begin{eqnarray*}
      N_0(x, n^0_k)=\prod_{i=1}^k m_i+\tau_0(x, n^0_k),
    \end{eqnarray*}
    where $\tau_0(x, n^0_k)$ is the number of zeros among the first non-fixed  digits until the position~$n^0_k$.

It is clear that      $N_0(x, n^0_k)=N_0(x, n^1_k)$ and
    $$\frac{N_0(x, n^0_k)}{n^0_k}=
    \frac{1+\tau_0(x, n^0_k)\Big(\prod\limits_{i=1}^k m_i\Big)^{-1}}{s+1+\frac{1}{m_k}\sum\limits_{i=1}^{k-2}\frac{1}{m_i}};
    $$

      $$ \frac{N_0(x, n^1_k)}{n^1_k}=\frac{1+\tau_0(x, n^0_k)\Big(\prod\limits_{i=1}^k m_i\Big)^{-1}}{s+1+\frac{1}{m_1}+\frac{1}{m_k}\sum\limits_{i=2}^{k-2}\frac{1}{m_i}}.
    $$

 If the limit $\lim\limits_{k\to\infty}\tau_0(x, n^0_k)\Big(\prod\limits_{i=1}^k m_i\Big)^{-1}$ does not exist, then the limit $\lim\limits_{k\to\infty}\frac{N_0(x,\; k)}{k}$ also does not exist.

 If the limit $\lim\limits_{k\to\infty}\tau_0(x, n^0_k)\Big(\prod\limits_{i=1}^k m_i\Big)^{-1}=a(x)$ exists, then
    $$\lim_{k\to\infty}\frac{N_0(x,\; n^0_k)}{n^0_k}=\frac{1+a(x)}{s+1} \quad\text{and}\quad
    \lim_{k\to\infty}\frac{N_0(x,\; n^1_k)}{n^1_k}=\frac{1+a(x)}{s+1+\frac 1{m_1}}.$$

Therefore, for any $x\in T_{s, \; l}$ the digit `$0$' does not have the frequency in the $Q_\infty$-ex\-pansion of~$x.$

Theorem~\ref{theorem on simple sufficient cond for non-faithfulness} shows that the family $\Phi(Q_\infty)$ could be non-faithful for the Hausdorff dimension calculation on the unit interval. Nevertheless we show that for any stochastic vector $Q_\infty$ the family $\Phi$ of all $Q_\infty$-cylinders is admissible for the Hausdorff dimension calculation of  $T_{s,l}$, i.\,e.,
$$\dim_H(T_{s, \; l})=\dim_H(T_{s, \; l}, \Phi).$$

To this end let us consider an arbitrary covering of $T_{s,l}$ by closed intervals $E_j=[a_j, b_j]$. Let  $I_j\!:=T_{s, \; l}\cap E_j.$

  Let $\Delta^j$ be  the cylinder of minimal length among all cylinders containing the set $I_j,$ and let $k_j$ be the rank of  $\Delta^j$.
Then~$(k_j+1)\in \mathrm{Flex}.$

Let $$\Delta^j = \Delta_{\alpha_1\ldots\alpha_{k_j}}=\bigcup_{i=0}^\infty \Delta_{\alpha_1\ldots\alpha_{k_j} i}.$$
Let $c_j\!:=\inf I_j,  d_j\!:=\sup I_j,$ then $c_j\in \Delta_{\alpha_1\ldots\alpha_{k_j} 0}.$
From the construction of the set $T_{s,l}$ it follows that $T_{s, \; l}\cap\Delta_{\alpha_1\ldots\alpha_{k_j} i}=\emptyset,$ $\forall i\geq l$ and $T_{s, \; l}\cap\Delta_{\alpha_1\ldots\alpha_{k_j} (l-1)}\ne\emptyset$. Since $d_j$ is the supremum of $I_j,$ we conclude  that $d_j\in\Delta_{\alpha_1\ldots\alpha_{k_j} (l-1)}.$ Hence, $\Delta_{\alpha_1\ldots\alpha_{k_j} 1}\subset [c_j, d_j],$ and, therefore
$$|(c_j, d_j)|>|\Delta_{\alpha_1\ldots\alpha_{k_j} 1}|=q_1\cdot|\Delta_{\alpha_1\ldots\alpha_{k_j}}|.$$
So, $|\Delta^j|< \frac{1}{q_1}|E_j|$. Thus, for any closed interval $E_j$ we can cover the set $T_{s, \; l}\cap E_j$ by one cylinder of length not larger than~$\frac{1}{q_1}|E_j|.$

Therefore, $$H^\alpha(T_{s, \; l})\leq H^\alpha(T_{s, \; l}, \Phi)\leq \frac{1}{q_1^\alpha}H^\alpha(T_{s, \; l})$$
for any $\alpha\in (0, 1], $ and, hence, $\dim_H(T_{s, \; l}, \Phi)=\dim_H(T_{s, \; l}).$

Finally, let us estimate the Hausdorff dimension of the set $T_{s,l}$. Our  purpose is  to prove that the
 Hausdorff dimension of the set  $T_{s, \; l}$ is close to 1 for large enough $s$ and $l$.
To this end we construct a special singularly continuous probability measure such that the set $T_{s, \; l}$ is the topological support of the measure.

Let $n_k$ be the number of the position at which the $k$-th group is ended.

 The number of non-fixed digits $j\in\mathrm{Flex}$ among the first $n_k$ digits is equal to
 $s\cdot\!\prod\limits_{i=1}^k m_i.$
 The number of a fixed digit $j$ ($j \in \{0, 1, \dots, k-1\}$) among the first $n_k$ digits is equal to~$\frac{\prod\limits_{i=1}^k m_i}{m_j}.$

   Let $\xi(l)$ be a random variable with independent $Q_\infty$-digits $\xi_k(l)$ defined by
  $$\xi(l)=\Delta_{\xi_1(l)\xi_2(l)\ldots\xi_k(l)\ldots},$$
  where $\xi_k(l)$ has the following distributions:
\begin{itemize}
  \item[$-$] if $k\in\mathrm{Fix}(j),$ then

 {\renewcommand{\arraystretch}{1.4}
\hspace{2cm}  \begin{tabular}{c|c}
  $\xi_k(l)$ & $j$            \\ \hline
  $~$      & $p_{jk}=1$ \\
\end{tabular},}
$j\in\mathbb{N}_0;$

\item[$-$] if $k\in\mathrm{Flex},$ then

{\renewcommand{\arraystretch}{1.4}
\hspace{2cm} \begin{tabular}{c|c|c|c|c}
  $\xi_k(l)$ & $0$   & $1$     & $\ldots$ & $l-1$       \\ \hline

  $~$      & $p_{0k}=\frac{q^{}_0}{S_l}$ & $p_{1k}=\frac{q_1}{S_l}$  & $\ldots$ & $p_{(l-1)k}=\frac{q_{l-1}}{S_l}$
\end{tabular},}

where $S_l\!:=\sum\limits_{i=0}^{l-1}q_i.$
\end{itemize}

Let $\mu_{\xi(l)}$ be the above defined probability distribution of the corresponding random variable $\xi(l)$ with independent $Q_\infty$-digits.
It is clear that  $T_{s, \; l}$ is the topological support of the measure $\mu_{\xi(l)}.$
So,
\begin{equation}\label{eq:ineqfordim}
  \dim_H(T_{s, l}) = \dim_H(T_{s, l}, \Phi)\geq \dim_H (\mu_{\xi(l)}, \Phi).
\end{equation}

Fine fractal properties of probability measures with independent $Q_\infty$-sym\-bols were studied in \cite{NT_TVIMS12}. In particular, it has been proven there that under the assumptions
$$\sum\limits_{k=1}^\infty\frac{\sum\limits_{i=0}^\infty p_{ik}\ln^2 p_{ik}}{k^2}<\infty ~~~\mbox{and}~~~ \sum\limits_{k=1}^\infty\frac{\sum\limits_{i=0}^\infty p_{ik}\ln^2 q_{i}}{k^2}<\infty,$$
the Hausdorff dimension of the measure with respect to $\Phi$ can be calculated as follows
  $$\dim_H(\mu_{\xi(l)}, \Phi)=\varliminf\limits_{n\to\infty}\frac{\sum\limits_{k=1}^{n}h_k}{\sum\limits_{k=1}^n b_k},$$
where $$ h_k\!:= -  \sum\limits_{i=0}^\infty p_{ik}\ln p_{ik}, \quad b_k\!:= -  \sum\limits_{i=0}^\infty p_{ik}\ln q_{i}.$$

  In our case  $\sum\limits_{k=1}^\infty\frac{\sum\limits_{i=0}^\infty p_{ik}\ln^2 p_{ik}}{k^2}<\infty,$ because
  $$\sum_{i=0}^\infty p_{ik}\ln^2 p_{ik}=\begin{cases}
    0, &\text{if}\; k\in \mathrm{Fix},\\
    \sum\limits_{i=0}^{l-1}\frac{q_i}{S_l}\ln^2\frac{q_i}{S_l}, &\text{if}\; k\in\mathrm{Flex}.
  \end{cases}
  $$
  Let us show that $\sum\limits_{k=1}^\infty\frac{\sum\limits_{i=0}^\infty p_{ik}\ln^2 q_{i}}{k^2}<\infty$ as well.

Since
  $$\sum_{i=0}^\infty p_{ik}\ln^2 q_{i}=\begin{cases}
    \ln^2 q_j, &\text{if}\; k\in \mathrm{Fix}(j),  ~~ j \in \mathbb{N}_0,\\
    \sum\limits_{i=0}^{l-1}\frac{q_i}{S_l}\ln^2{q_i}, &\text{if}\; k\in\mathrm{Flex},
  \end{cases}
  $$
  it is enough to prove the convergence of the series $\sum\limits_{j=0}^\infty\left(\sum\limits_{i\in\mathrm{Fix}(j)}\frac{\ln^2 q_j}{i^2}\right).$

From the construction of $T_{s,l}$ it follows that $m_j$ is less than the minimal element of the set $\mathrm{Fix}(j).$ Then
$$ \sum\limits_{j=0}^\infty\left(\sum\limits_{i\in\mathrm{Fix}(j)}\frac{\ln^2 q_j}{i^2}\right)<\sum\limits_{j=0}^\infty\left(\sum\limits_{i =m_j+1}^\infty\frac{\ln^2 q_j}{i(i-1)}\right)= \sum_{j=0}^\infty\frac{\ln^2 q_j}{m_j} <\infty.
$$

Since $h_k \leq b_k$, it is not hard to check that $\varliminf\limits_{n\to\infty}\frac{\sum\limits_{i=1}^{n}h_i}{\sum\limits_{i=1}^n b_i} = \varliminf\limits_{k\to\infty}\frac{\sum\limits_{i=1}^{n_k}h_{i}}{\sum\limits_{i=1}^{n_k} b_{i}}.
$

Let $\mathbf{h}_l\!:=\sum\limits_{i=0}^{l-1}\frac{q_i}{S_l}\ln\frac{S_l}{q_i}, \quad \mathbf{b}_l
\!:=\sum\limits_{i=0}^{l-1}\frac{q_i}{S_l}\ln\frac{1}{q_i}.$ Then
$$\sum_{i=1}^{n_k} h_i = s \mathbf{h}_l \prod\limits_{i=1}^{k}m_i \quad\mbox{and} \quad \sum_{i=1}^{n_k} b_i = s \mathbf{b}_l \prod\limits_{i=1}^{k}m_i + \prod\limits_{i=1}^k m_i\Big(\sum\limits_{j=0}^{k-1} \frac{1}{m_j}\ln\frac{1}{q_j}\Big).$$

So,
$$
  \frac{\sum\limits_{i=1}^{n_k} h_i}{\sum\limits_{i=1}^{n_k} b_i}=\frac{s \mathbf{h}_l}{s\mathbf{b}_l+\ln\frac{1}{q_0}+\sum\limits_{i=1}^{k-1}\frac{1}{m_i}\ln\frac{1}{q_i}}= \frac{s \mathbf{h}_l}{s\mathbf{b}_l+\ln\frac{1}{q_0}+\sum\limits_{i=1}^{k-1}\frac{1}{{[\ln^2\frac{1}{q_{i}}]2^i}}\ln\frac{1}{q_i}}.
$$

Let $K := \lim\limits_{k\to\infty}\left(\ln\frac{1}{q_0}+\sum\limits_{i=1}^{k-1}\frac{\ln\frac{1}{q_i}}{[\ln^2\frac{1}{q_{i}}]2^i}\right)<\infty.$
Then $\lim\limits_{k\to\infty}\frac{\sum\limits_{i=1}^{n_k} h_i}{\sum\limits_{i=1}^{n_k} b_i}=\frac{s\mathbf{h}_l}{s\mathbf{b}_l+K}.$

So, taking into account \eqref{eq:ineqfordim}, we get $\dim_H T_{s, l}\geq \frac{s \mathbf{h}_l}{s\mathbf{b}_l+K}.$

Since $S_l\to1 \;(l\to\infty)$,
we get
$$\dim_H (L(Q_\infty))\geq \sup_{s, l}\dim_H T_{s, l}=1,$$ which proves the theorem.
\end{proof}

%%%%%
%%%%%
%%%%%Properties of subsets of the set of non-normal numbers defined  in $s$-adic expansion were intensively studied during a long time (see, i.\,e. \cite{ BSS02, APT3, Bisbas12, Ols2, Ols1, Ols4, Ols3,  PT} and references  therein). Here we consider a similar questions for the set of numbers defined by $Q_\infty$-expansion. The difference from earlier works in the infinity of the alphabet, from which comes a series of new phenomena.
%%%%%
%%%%%                                                              In \cite{BI09, FLMW10} studied the Hausdorff dimension of a concrete class of sets defined in terms
%%%%%of the frequencies of digits in expansion with infinite alphabet, more precisely, sets of numbers with prescribed frequencies of digits (Besicovitch--Eggleston sets).
%%%%%
%%%%%One of our purposes is to determine the Hausdorff dimension of the set of real numbers having no frequencies of all digits in $Q_\infty$-expansion (essentially non-normal numbers). In this paper we sequel investigation started in \cite{AKNT2, NT_TVIMS12}.

\bigskip
\textbf{Acknowledgment}

This work was partly supported by  SFB-701 ``Spectral Structures and Topological Methods in Mathematics'' (Bielefeld University), STREVCOM FP-7-IRSES 612669 project and by the Alexander von Humboldt Foundation.
The authors would like to express their gratitude to Prof.~Yuval Peres (Microsoft Research) for fruitful discussions and valuable remarks on problems related to the non-faithfulness of coverings.

\end{document}